\documentclass[11pt]{amsart}

\usepackage{amsfonts,latexsym}
\usepackage{amsmath}
\usepackage{amstext}
\usepackage{amssymb}
\usepackage{color}
\usepackage{graphicx}
\graphicspath{{figures/}}

\newcommand{\R}{\mathbb{R}}

\newtheorem{theorem}{Theorem}[section]
\newtheorem{lemma}[theorem]{Lemma}
\newtheorem{proposition}[theorem]{Proposition}

\theoremstyle{definition}
\newtheorem{definition}[theorem]{Definition}
\newtheorem{problem}[theorem]{Problem}
\newtheorem{example}[theorem]{Example}
\theoremstyle{remark}
\newtheorem{remark}[theorem]{Remark}

\numberwithin{equation}{section}

\begin{document}

\title[Controller Synthesis for Robust Invariance of Polynomial Dynamical Systems]
      {Controller Synthesis for Robust Invariance of Polynomial Dynamical Systems using Linear Programming}
\thanks{This work was supported by the Agence Nationale de la Recherche (VEDECY project - ANR 2009 SEGI 015 01).}


\author[Mohamed Amin Ben Sassi]{Mohamed Amin Ben Sassi}
\address{Laboratoire Jean Kuntzmann \\
Universit\'e de Grenoble \\
B.P. 53, 38041 Grenoble, France}
\email{Mohamed-Amin.Bensassi@imag.fr}

\author[Antoine Girard]{Antoine Girard}
\address{Laboratoire Jean Kuntzmann \\
Universit\'e de Grenoble \\
B.P. 53, 38041 Grenoble, France} \email{Antoine.Girard@imag.fr}


\maketitle


\begin{abstract}
In this paper, we consider a control synthesis problem for a class of polynomial dynamical systems
subject to bounded disturbances and with input constraints. 
More precisely, we aim at synthesizing at the same time a controller and an invariant set 
for the controlled system under all admissible disturbances. 
We propose a computational method to solve this problem. 
Given a candidate polyhedral invariant, we show that controller synthesis can be formulated as 
an optimization problem involving polynomial cost functions over bounded polytopes for which effective linear programming relaxations can be obtained. Then, we propose an iterative approach
to compute the controller and the polyhedral invariant at once. Each iteration of the approach
mainly consists in solving two linear programs (one for the controller and one for the invariant) and is thus computationally tractable.
Finally, we show with several examples the usefulness of our method in applications.
\end{abstract}

\section{Introduction}

The design of nonlinear systems remains a challenging problem in control science. 
In the past decade, building on spectacular breakthroughs in optimization over polynomial functions~\cite{Lasserre2001,Parrilo2003},
several computational methods have been developed for synthesizing controllers for polynomial dynamical systems~\cite{Prajna2004,Lasserre2008}. These approaches have shown successful for several synthesis problems 
such as stabilization or optimal control in which Lyapunov functions and cost functions can be represented or approximated by polynomials. However, these approaches are not suitable for 
some other problems such as those involving polynomial dynamical systems with constraints on states and inputs, and subject to bounded disturbances. 

In this paper, we consider a control synthesis problem for this class of systems.
More precisely, given a polynomial dynamical system with input constraints and bounded disturbances,
given a set of initial states $\underline{P}$ and a set of safe states $\overline P$, we aim at synthesizing a controller satisfying the input constraints and such that trajectories starting in $\underline P$ remain in $\overline P$
for all possible disturbances. This problem can be solved by computing jointly the controller and an invariant set 
for the controlled system which contains $\underline P$ and is included in $\overline P$ (see e.g.~\cite{blanchini99}).
We propose a computational method to solve this problem. 
We use parameterized template expressions for the controller and the invariant.
Given a candidate polyhedral invariant, we show that controller synthesis can be formulated as 
an optimization problem involving polynomial objective functions over bounded polytopes.
Recently, using various tools such as the blossoming principle~\cite{Ramshaw89}
for polynomials, multi-affine functions~\cite{Belta06} and Lagrangian duality,
it has been shown how effective linear programming relaxations can be obtained for 
such optimization problems~\cite{Bensassi2010}. We then propose an iterative approach
to compute jointly a controller and a polyhedral invariant. Each iteration of the approach
mainly consists in solving two linear programs and is thus computationally tractable.
Finally, we show applications of our approach to several examples.

\section{Problem Formulation}

In this work, we consider a  nonlinear affine control system subject to input constraints and bounded disturbances:
\begin{equation}
\label{eq:odeg}
\dot x(t) = f(x(t),d(t))+g(x(t),d(t)) u(t),\; d(t) \in D\;, u(t)\in U
\end{equation}
where $x(t)\in R_X \subseteq \R^n$ denotes the state of the system, $d(t)\in D\subseteq \R^m$ is an external disturbance and
$u(t) \in U \subseteq \R^p$ is the control input. 
We assume that the vector field $f:\R^n \times \R^m\rightarrow \R^n$ and 
the control matrix $g:\R^n \times \R^m\rightarrow \R^{(n \times p)}$, defining the dynamics of the system, 
are multivariate polynomial maps. We also assume that the set of states is a bounded rectangular domain:
$R_X=[\underline{x_1},\overline{x_1}] \times \dots \times [\underline{x_n},\overline{x_n}]$ with 
$\underline{x_k}<\overline{x_k}$ for all $k\in \{1,\dots,n\}$; and
that the set of disturbances $D$ and the set of inputs 
$U$ are convex compact polytopes:
$$
D=\left\{d \in \R^m |\; \alpha_{D,k}\cdot d \le \beta_{D,k} ,\; \forall k\in \mathcal K_D \right\} \text{ and }
U=\left\{u \in \R^p |\; \alpha_{U,k}\cdot u \le \beta_{U,k} ,\; \forall k\in \mathcal K_U \right\}
$$
where $\alpha_{D,k}\in \R^m$, $\beta_{D,k}\in \R$,
$\alpha_{U,k}\in \R^p$, $\beta_{U,k} \in \R$,
$\mathcal K_D$ and $\mathcal K_U$ are finite sets of indices. 
We will denote by $R_D=[\underline{d_1},\overline{d_1}] \times \dots \times [\underline{d_m},\overline{d_m}]$ the interval hull of polytope $D$, that is the smallest rectangular domain containing $D$; and by
$V_X=\{\underline{x_1},\overline{x_1}\} \times \dots \times \{\underline{x_n},\overline{x_n}\}$ and 
$V_D=\{\underline{d_1},\overline{d_1}\} \times \dots \times \{\underline{d_m},\overline{d_m}\}$ the set of 
vertices of $R_X$ and $R_D$.
The present work deals with controller synthesis for a notion of invariance defined as follows:
\begin{definition}
Consider a set of states $P\subseteq R_X$ and a controller $h:R_X\rightarrow U$, the controlled system 
\begin{equation}
\label{eq:odec}
\dot x(t) = f(x(t),d(t))+g(x(t),d(t)) h(x(t)),\; d(t) \in D, 
\end{equation}
is said to be $P$-invariant if all trajectories
with $x(0)\in P$ satisfy $x(t)\in P$ for all $t\ge 0$. 
\end{definition}

Let us remark that this is a notion of robust invariance since it has to hold for all possible disturbances. 
Let $\underline P \subseteq \overline P \subseteq R_X$ be convex compact polytopes.
In this paper, we consider the problem of synthesizing a
controller $h$ for system~(\ref{eq:odeg}) such that all controlled trajectories
starting in $\underline P$ remain in $\overline P$ forever. This can be seen as a safety property where 
$\underline P$ is the set of initial states and
$\overline P$ is the set of safe states.
The problem can be solved synthesizing jointly a controller and a polyhedral invariant $P\subseteq R_X$ containing 
$\underline P$ and included in $\overline P$:
\begin{problem}\label{prob}
Synthesize a controller $h: R_X \rightarrow U$ and a convex compact polytope $P$
such that $\underline P \subseteq P \subseteq \overline P$  and the controlled system (\ref{eq:odec})
is $P$-invariant.
\end{problem}

In the following, we describe an approach to solve this problem. To restrict the search space, we shall use parameterized
template expressions for the controller $h$ and the invariant $P$.
Firstly, we will impose the orientation of the facets of polytope $P$ by choosing normal vectors 
in the set $\{\gamma_k \in \R^n |\; k\in \mathcal K_X\}$ where $\mathcal K_X$ is a finite set of indices.
Then, polytope $P$ can be written under the form 
$$
P=\left\{x\in \R^n |\; \gamma_k\cdot x \le \eta_k ,\; \forall k\in \mathcal K_X \right\}
$$
where the vector $\eta \in \R^{|\mathcal K_X|}$, to be determined, specifies the position of the facets. The facets of $P$ are denoted by $F_k$ for $k\in \mathcal K_X$, where
$
F_k=\left\{x\in \R^n |\; \gamma_k\cdot x =\eta_k \text{ and } \gamma_i\cdot x \le \eta_i ,\; \forall i\in \mathcal K_X\setminus \{k\} \right\}.
$
For simplicity, we will assume that the polytopes $\underline P$ and $\overline P$ are of the form:
$
\underline P =\{x\in \R^n |\; \gamma_k\cdot x \le \underline{\eta}_k ,\; \forall k\in \mathcal K_X \} \text{ and }
\overline P =\{x\in \R^n |\; \gamma_k\cdot x \le \overline{\eta}_k ,\; \forall k\in \mathcal K_X \}.
$
Then, the condition $\underline P \subseteq P \subseteq \overline P$ translates to 
$\underline{\eta}_k \le \eta_k \le \overline{\eta}_k$, for $k\in \mathcal K_X$. 
Secondly, we will search the controller $h$ in a subspace spanned by a polynomial matrix:
$$
h(x)=H(x) \theta
$$
where $\theta \in \R^q$ is a parameter to be determined and the matrix $H:\R^n \rightarrow \R^{(p \times q)}$ is a given multivariate polynomial map. 
The use of a template expression is natural when searching for a controller with a particular structure.
The input constraint (i.e. for all $x\in R_X$, $h(x) \in U$) is then equivalent to
\begin{equation}
\label{eq:inputcons}
 \forall k\in \mathcal K_U,\; \forall x\in R_X,\; \alpha_{U,k}\cdot H(x)\theta \le \beta_{U,k}.
\end{equation}

Under these assumptions, the dynamics of the controlled system~(\ref{eq:odec}) can be rewritten under the form
$$
\dot x(t) = f(x(t),d(t))+G(x(t),d(t)) \theta,\; d(t) \in D, 
$$
where the matrix of polynomials $G(x,d)=g(x,d)H(x)$. 
From the standard characterization of invariant sets (see~\cite{aubin1991}), it follows that 
the controlled system (\ref{eq:odec}) is $P$-invariant if and only if
\begin{equation}
\label{eq:car}
 \forall k\in \mathcal K_X,\; \forall x\in F_k, \forall d\in D,\; \gamma_k\cdot (f(x,d)+G(x,d) \theta) \le 0.
\end{equation}
Then, Problem~\ref{prob} can be solved by computing vectors $\theta \in \R^q$ and $\eta \in \R^{|\mathcal K_X|}$ with
$\underline{\eta}_k \le \eta_k \le \overline{\eta}_k$ for all $k\in \mathcal K_X$,  and such that
(\ref{eq:inputcons}) and (\ref{eq:car}) hold.
In the following, we first show how, given the vector $\eta \in \R^{|\mathcal K_X|}$
(and hence the polytope $P$),
we can compute, using linear programming, the parameter $\theta$ (and hence the controller $h$) 
such that the controlled system~(\ref{eq:odec}) is $P$-invariant.
Then, we show how to compute jointly the controller $h$ and the polytope $P$ using an iterative approach 
based on sensitivity analysis of linear programs. Before that, we shall review some recent results on linear relaxations
for optimization of polynomials over bounded polytopes~\cite{Bensassi2010}.

\section{Optimization of Polynomials over Polytopes}
\label{sec:opt}
In this section, we review some recent results of~\cite{Bensassi2010} that will be useful for solving Problem~\ref{prob}.
Let us consider the following optimization problem involving a polynomial on a bounded polytope:
\begin{equation}
\label{eq:opt}
\begin{array}{llr}
\text{minimize} & c \cdot p(y)\\
\text{over} & y\in R, \\
\text{subject to} 
& a_i \cdot y \le b_i, & i\in I, \\
& a_j \cdot y = b_j, & j\in J,
\end{array}
\end{equation}
where $p:\R^m \rightarrow \R^n$ is a multivariate polynomial map, $c \in \R^n$,
$R=[\underline{y_1},\overline{y_1}] \times \dots \times [\underline{y_m},\overline{y_m}]$ is a rectangle of $\R^m$; 
$I$ and $J$ are finite sets of indices; 
$a_k\in \R^m$, $b_k\in \R$, for all $k\in I\cup J$.
Let us remark that even though the polytope defined by the constraints indexed by $I$ and $J$ is unbounded 
in $\R^m$, the fact that we consider $y\in R$ which is a bounded rectangle of $\R^m$ 
results in an optimization problem on a bounded (not necessarily full dimensional) polytope of $\R^m$. 
Let $p^*$ denote the optimal value of problem~(\ref{eq:opt}).
The approach presented in~\cite{Bensassi2010} allows us to compute a guaranteed lower bound $d^*$ of $p^*$.
The approach is as follows. First, using the so-called blossoming principle~\cite{Ramshaw89}, we transform problem~(\ref{eq:opt})
into an  equivalent optimization problem involving a multi-affine function on a polytope. The dual of this problem 
is then a linear program easily solvable and whose optimal value is a guaranteed lower bound of $p^*$.

\subsection{Blossoming principle}

Multi-affine functions form a particular class of multivariate polynomials.
Essentially, a multi-affine function is a function which is affine in each of its variables when the other variables are regarded as constant. 
\begin{definition} A multi-affine function $q:\R^M \rightarrow \R$ is a multivariate polynomial in the variables $z_1,\dots,z_M$ where the degree of $g$ in each of its 
variables is at most $1$:
$$
q(z)=q(z_1,\dots,z_M)=\sum_{(d_1,\dots,d_M)\in \{0,1\}^M } q_{(d_1,\dots,d_M)} z_1^{d_1}\dots z_M^{d_M}
$$
where $q_{(d_1,\dots,d_M)} \in \R$ for all $(d_1,\dots,d_M)\in \{0,1\}^M$.
A map $q:\R^M\rightarrow \R^n$ is a multi-affine map if each of its components is a multi-affine function.
\end{definition}

It is shown in~\cite{Belta06} that a multi-affine function $q$ is uniquely 
determined by its values at the vertices of a rectangle $R'$ of $\R^M$. 
Moreover, for all $x\in R'$, $q(x)$ is a convex combination 
of the values at the vertices so that we have the following result:
\begin{lemma}\label{lem:conv}
Let $q:\R^M \rightarrow \R$ be a multi-affine function and $R'$ a rectangle of $\R^M$ with set of vertices $V'$, then
$
\displaystyle{\min_{x\in R'} q(x)= \min_{v\in V'} q(v).}
$
\end{lemma}

The blossoming principle (see e.g.~\cite{Ramshaw89}) consists in mapping the set of polynomial maps to the set of multi-affine maps as follows.
Let $p:\R^m\rightarrow \R^n$ be a polynomial map. Let $\delta_1,\dots,\delta_m$ denote the degree of 
$p$ in the variables $y_1,\dots,y_m$ respectively. 
Let $\Delta=\{0,\dots,\delta_1\} \times \dots \times \{0,\dots,\delta_m\}$, then for all $y\in \R^m$,
$p(y)=(p_1(y),\dots,p_n(y))$, where for all $j=1,\dots,n$, the components $p_j:\R^m \rightarrow \R$ are 
multivariable polynomial functions that can be written under the form:
$$
p_j(y)=p_j(y_1,\dots,y_m)=\sum_{(d_1,\dots,d_m)\in \Delta } p_{j,(d_1,\dots,d_m)} y_1^{d_1}\dots y_m^{d_m}
$$
where $p_{j,(d_1,\dots,d_m)} \in \R$, for all $(d_1,\dots,d_m)\in \Delta$ and $j=1,\dots,n$.

\begin{definition}\label{def:blossom}
The blossom of the polynomial map $p:\R^m \rightarrow \R^n$ is the map
$q:\R^{\delta_1+\dots+\delta_m} \rightarrow \R^n$ whose components are given for $j=1,\dots,n$ and 
 $z=(z_{1,1},\dots,z_{1,\delta_1},\dots,z_{m,1},\dots,z_{m,\delta_m})\in \R^{\delta_1+\dots+\delta_m}$ by
$$
q_j(z)=\sum_{(d_1,\dots,d_m)\in \Delta } p_{j,(d_1,\dots,d_m)} B_{d_1,\delta_1}(z_{1,1},\dots,z_{1,\delta_1})\dots
B_{d_m,\delta_m}(z_{m,1},\dots,z_{m,\delta_m})
$$
with
$$
B_{d,\delta}(z_1,\dots,z_\delta) = \frac{1}{\left(\begin{smallmatrix} \delta \\ d \end{smallmatrix} \right)} \sum_{\sigma \in C(d,\delta)} 
z_{\sigma_1} \dots z_{\sigma_d}
$$
where $C(d,\delta)$ denotes the set of combinations of $d$ elements in  $\{1,\dots,\delta\}$.
\end{definition}

\begin{example} The blossom of the polynomial map $p(x_1,x_2)=(x_1+x_2^3,x_1^2x_2^2)$ is $q:\R^5 \rightarrow \R^2$ whose components are given by
$$
\begin{array}{rll}
q_1(z_{1,1},z_{1,2},z_{2,1},z_{2,2},z_{2,3})&=&\frac{1}{2} (z_{1,1}+z_{1,2})+z_{2,1}z_{2,2}z_{2,3} \\
q_2(z_{1,1},z_{1,2},z_{2,1},z_{2,2},z_{2,3})&=&z_{1,1}z_{1,2}\frac{1}{3}(z_{2,1}z_{2,2}+z_{2,2}z_{2,3}+z_{2,3}z_{2,1})
\end{array}
$$
\end{example}

From Definition~\ref{def:blossom}, it follows that the blossom $q$ of the polynomial map $p$
satisfies the following properties~\cite{Ramshaw89}:
\begin{enumerate}
\item It is a multi-affine map.

\item It satisfies the diagonal property:
$
q(z_1,\dots,z_1,\dots,z_m,\dots,z_m) = p(z_1,\dots,z_m).
$

\item Let $z,z' \in \R^{\delta_1+\dots+\delta_m}$, with 
$z=(z_{1,1},\dots,z_{1,\delta_1},\dots,z_{m,1},\dots,z_{m,\delta_m})$  and 
$z'=(z'_{1,1},\dots,z'_{1,\delta_1},$ $\dots,z'_{m,1},\dots,z'_{m,\delta_m}),$ 
we denote
$z \cong z'$ if, for all $j=1,\dots,m$, there exists a permutation $\pi_j$
such that $(z_{j,1},\dots,z_{j,\delta_j})= \pi_j(z'_{j,1},\dots,z'_{j,\delta_j})$.
It is easy to see that $\cong$ is an equivalence relation. Moreover, 
for all $z\cong z'$, $q(z)=q(z')$.
\end{enumerate}
The diagonal property clearly allows us to recast problem (\ref{eq:opt}) for a multivariate polynomial map $p$ as a 
problem involving its blossom $q$ subject to inequality and equality constraints:
\begin{equation}
\label{eq:opt1}
\begin{array}{llr}
\text{minimize } & c \cdot q(z)\\
\text{over } & z\in R', \\
\text{subject to } 
& {a_i}' \cdot z \le b_i, & i \in I \\
& a_j' \cdot z = b_j, & j\in J, \\
& z_{k,l}=z_{k,l+1} ,& \; k=1,\ldots,n,\; l=1,\ldots,\delta_{j}-1.
\end{array}
\end{equation}
where $R'=[\underline{y_1},\overline{y_1}]^{\delta_1} \times \dots \times [\underline{y_m},\overline{y_m}]^{\delta_m}$ 
and the vectors $a_k'$ are given for $k\in I \cup J$ by
$$
a_k'={(\frac{a_{k,1}}{\delta_1},\dots,\frac{a_{k,1}}{\delta_1},\dots,\frac{a_{k,m}}{\delta_m},\dots,\frac{a_{k,m}}{\delta_m})}.
$$ 

\subsection{Linear programming relaxation}

In~\cite{Bensassi2010}, based on Lemma~\ref{lem:conv}, it is shown that the Lagrangian dual of problem~(\ref{eq:opt1}) is actually a linear program. After some complexity reduction enabled by the properties of blossoms stated above, the following result can be stated:
\begin{theorem}[\cite{Bensassi2010}]
\label{th:dual}
The dual of problem (\ref{eq:opt1}) is equivalent to the following linear program:
\begin{equation}
\label{eq:dual}
\begin{array}{llr}
\text{maximize} & t\\
\text{over} & t\in \R,\; \lambda\in \R^{|I|},\; \mu \in \R^{|J|}\\
\text{subject to} &\lambda_i \ge 0, & i\in I \\ 
&t \le c \cdot q(\overline{v})+\displaystyle{\sum_{i\in I} \lambda_i (a_i' \cdot \overline{v} - b_i)}+
\displaystyle{\sum_{j\in J} \mu_j (a_j' \cdot \overline{v} - b_j)} , & \overline{v}\in \overline{V'}.
\end{array}
\end{equation}
where $\overline{V'}=V'/\cong$ with  
$V'=\{\underline{y_1},\overline{y_1}\}^{\delta_1} \times \dots \times \{\underline{y_m},\overline{y_m}\}^{\delta_m}$. 
Moreover the optimal value $d^*$ of this linear program
is a guaranteed lower bound of the optimal value $p^*$ of problem (\ref{eq:opt}).
\end{theorem}
 
The previous theorem provides a simple and efficient way to compute a guaranteed lower bound of the minimal value of a polynomial on a bounded 
polytope. In the following section, we will show how this can be used to solve our controller synthesis problem.

\section{Controller Synthesis}
In this section, given polytope $P=\left\{x\in \R^n |\; \gamma_k\cdot x \le \eta_k ,\; \forall k\in \mathcal K_X \right\}$,
we show how to design a controller $h:R_X\rightarrow U$ such that the system (\ref{eq:odec}) is $P$-invariant.
As explained before, we search  the controller  in a subspace spanned by a polynomial matrix:
$h(x)=H(x) \theta$ where $\theta \in \R^q$.
Let $F(x,d)=f(x,d)+G(x,d)\theta$, it is a polynomial of degree $ \delta_1,\dots,\delta_n,\rho_1,\dots,\rho_m$ 
in the variables $x_1,\dots,x_n,d_1,\dots,d_m$.
Its blossom is $F_b=f_b +G_b \theta$ where $f_b$ and $G_b$ are the blossoms of $f$ and $G$
regarded as polynomials of degree $ \delta_1,\dots,\delta_n,\rho_1,\dots,\rho_m$.
Let $H_b$ be the blossom of the matrix $H$ regarded as a polynomial map of degree $\delta_1,\dots,\delta_n$
in the variables $x_1,\dots,x_n$. Let 
$R_X'=[\underline{x_1},\overline{x_1}]^{\delta_1} \times \dots \times [\underline{x_n},\overline{x_n}]^{\delta_n}$,
$V_X'=\{\underline{x_1},\overline{x_1}\}^{\delta_1} \times \dots \times \{\underline{x_n},\overline{x_n}\}^{\delta_n}$,
$V_D'=\{\underline{d_1},\overline{d_1}\}^{\rho_1} \times \dots \times \{\underline{d_m},\overline{d_m}\}^{\rho_m}$,
$\overline{V_X'}={V_X'}/\cong$ and $\overline{V_D'}={V_D'}/\cong$.

We first establish sufficient conditions such that $h(x) \in U$ for all $x\in R_X$:
\begin{lemma}
\label{lem:cond}
If for all $l \in \mathcal{K_U}$, for all $\overline{v_X}\in \overline{V_X'}$,
$\alpha_{U,l} \cdot H_b(\overline{v_X})\theta \le \beta_{U,l}$, then  for all $x\in R_X$, $h(x) \in U$.
\end{lemma}
\begin{proof}
If for all $l \in \mathcal{K_U}$, for all $\overline{v_X}\in \overline{V_X'}$,
$\alpha_{U,l} \cdot H_b(\overline{v_X})\theta \le \beta_{U,l}$, then using the third property of the blossom 
we have for all $l \in \mathcal{K_U}$, for all ${v_X}\in {V_X'}$,
$\alpha_{U,l} \cdot H_b(v_X)\theta \le \beta_{U,l}$.
Since $H_b$ is a multi-affine map, Lemma \ref{lem:conv} implies that 
for all $l \in \mathcal{K_U}$, for all $z\in R_X'$,
$\alpha_{U,l} \cdot H_b(z)\theta \le \beta_{U,l}$.
Then, using the diagonal property of the blossom, we obtain 
for all $l \in \mathcal{K_U}$, for all $x\in R_X$,
$\alpha_{U,l} \cdot H(x)\theta \le \beta_{U,l}$
which is equivalent to say that for all $x\in R_X$, $h(x)\in U$. 
\end{proof}

The previous result gives a finite set of linear constraints which must be satisfied by parameter $\theta$.
We now establish conditions ensuring that the polytope $P$ is an invariant for system~(\ref{eq:odec}).
Let $k \in \mathcal K_X$, we will say that facet $F_k$ of the polytope $P$ is {\it blocked}
if for all $x\in F_k$, for all $d\in D$, $\gamma_k\cdot (f(x,d)+G(x,d) \theta) \le 0$.
It is clear that system~(\ref{eq:odec}) is $P$-invariant  if and only if all facets are blocked.

\begin{lemma}
\label{prop:block}
 Let $\theta \in \R^q$ and $ k\in \mathcal K_X$, then the facet $F_k$
is blocked if and only if the optimal value $p_k^*(\theta)$ of the following optimization problem is non negative:
\begin{equation}
\label{eq:primal1}
\begin{array}{llr}
\text{minimize} & -\gamma_k\cdot (f(x,d)+G(x,d)\theta)\\
\text{over} & x\in R_X, d\in R_D \\
\text{subject to} 
& \alpha_{D,j} \cdot d \le \beta_{D_j}, & j\in \mathcal K_D, \\
& \gamma_i \cdot x \le \eta_i, & i\in \mathcal K_X\setminus \{k\}, \\
& \gamma_k \cdot x = \eta_k.
\end{array}
\end{equation}
A guaranteed lower bound $d_k^*(\theta)$ of $p_k^*(\theta)$ 
is given by the optimal value of the following linear program:
\begin{equation}
\label{eq:block}
\begin{array}{lll}
\text{maximize} & t\\
\text{over} & t\in \R,\; \lambda^k\in \R^{|\mathcal K_X|},\tilde{\lambda}^k\in \R^{|\mathcal K_D|},\\
\text{subject to} &\displaystyle{\lambda_i^k} \ge 0, & i\in \mathcal K_X\setminus \{k\}, \\ 
&{\tilde{\lambda}^k}_j\ge 0, & j\in \mathcal K_D,\\
&t \le -\gamma_k \cdot (f_b(\overline{v_X},\overline{v_D})+ G_b(\overline{v_X},\overline{v_D})\theta) 
 \\
&\; \; \; \; \; +\displaystyle{\sum_{i=1}^{|\mathcal K_X|} \displaystyle{\lambda_i^k} ({\gamma_i}' \cdot \overline{v_X} - \eta_i)} +\displaystyle{\sum_{j=1}^{|\mathcal K_D|} {\tilde{\lambda}^k}_j (\alpha_{D,j}'\cdot \overline{v_D}  - \beta_{D,j})} ,& 
\overline{v_X}\in \overline{{V_X}'}
,\; \overline{v_D} \in \overline{{V_D}'}.
\end{array}
\end{equation}
where for all $i \in \mathcal K_X$ and all $j \in \mathcal K_D$ vectors ${\gamma_i}'$ and $\alpha_{D,j}'$ are given by:
$$
{\gamma_i}'={(\frac{\gamma_{i,1}}{\delta_1},\dots,\frac{\gamma_{i,1}}{\delta_1},\dots,\frac{\gamma_{i,n}}{\delta_n},\dots,\frac{\gamma_{i,n}}{\delta_n})},\;
\alpha_{D,j}'={(\frac{\alpha_{D,j,1}}{\rho_1},\dots,\frac{\alpha_{D,j,1}}{\rho_1},\dots,\frac{\alpha_{D,j,m}}{\rho_m},\dots,\frac{\alpha_{D,j,m}}{\rho_m})}.
$$ 
\end{lemma}
\begin{proof}
Remarking that $F_k=R_X \cap F_k$ and $D=R_D \cap D$, the first part of the Proposition is obvious.\\
For the second part, let us remark that from the definition of the equivalence relation $\cong$, we have
$(V_X' \times V_D')/\cong$ that is the same as $\overline{{V_X}'} \times 
\overline{{V_D}'}$. Then, we have just 
to apply the approach described in Section~\ref{sec:opt} where $y=(x,d)$ and the multivariate polynomial
 $p(y)$ is equal to $f(x,d)+G(x,d)\theta$.
 \end{proof}

Now we show how to choose $\theta \in \R^q $ such that the associated controller satisfy
for all $x\in R_X$, $h(x)\in U$ and the controlled system (\ref{eq:odec}) is $P$-invariant. 
\begin{proposition}
\label{prop:inv}
Let $d^*$ and $\left(t^*, ({\lambda^{k*}})_{k\in\mathcal K_X},({\tilde{\lambda}^{k*}})_{k\in\mathcal K_D},\theta^*\right)$ be the optimal value and an optimal solution  
of the following linear program:
\begin{equation}
\label{eq:inv}
\begin{array}{lll}
\text{maximize} & t\\
\text{over} & t\in \R,\; {\lambda}^k\in \R^{|\mathcal K_X|},\tilde{\lambda}^k\in \R^{|\mathcal K_D|},{\theta} \in \R^q,& k\in \mathcal{K_X}\\
\text{subject to} & \displaystyle{\lambda_i^k} \ge 0, &k\in \mathcal{K_X}, \quad i\in \mathcal K_X\setminus \{k\}, \\ 
&{\tilde{\lambda}^k}_j\ge 0,&k\in \mathcal{K_X},\quad j\in \mathcal K_D,\\
&\alpha_{U,l} \cdot H_b(\overline{v_X})\theta \le \beta_{U,l} , &l \in \mathcal{K_U},\quad \overline{v_X}\in \overline{V_X'},\\
&t \le -\gamma_k \cdot (f_b(\overline{v_X},\overline{v_D})+ G_b(\overline{v_X},\overline{v_D})\theta) 
 \\
&\; \; \; \; \;+\displaystyle{\sum_{i=1}^{|\mathcal K_X|} \displaystyle{\lambda_i^k} ({\gamma_i}' \cdot \overline{v_X} - \eta_i)}
+\displaystyle{\sum_{j=1}^{|\mathcal K_D|} {\tilde{\lambda}^k}_j (\alpha_{D,j}'\cdot \overline{v_D}  - \beta_{D,j})} ,&k\in \mathcal{K_X},  \overline{v_X}\in \overline{{V_X}'}
,\; \overline{v_D} \in \overline{{V_D}'}.
\end{array}
\end{equation}
Then, if $d^*$ is positive, the controller $h(x)=H(x)\theta^*$ satisfy for all $x\in R_X$, $h(x)\in U$ and
the controlled system~(\ref{eq:odec}) is $P$-invariant.
\end{proposition}
\begin{proof}
We first start by remarking that problem~(\ref{eq:inv}) is equivalent to the following optimization problem:
\begin{equation}
\label{eq:theta}
 \begin{array}{lll}
\text{maximize} & t\\
\text{over} & t\in \R,\; {\theta} \in \R^q\\
\text{subject to} &\alpha_{U,l} \cdot H_b(\overline{v_X})\theta \le \beta_{U,l} , &l \in \mathcal{K_U},\quad \overline{v_X}\in \overline{V_X'}, \\
& t \le d_k^*(\theta), & k\in \mathcal{K_X}
\end{array}
\end{equation}
where $d_k^*(\theta)$ is the optimal value of linear program~(\ref{eq:block}).
Then, if $d^*\ge 0$, this means that for the optimal $\theta^*$, we have 
for all $k\in \mathcal K_X$, $d_k^*(\theta^*) \ge 0$. Therefore, by Lemma~\ref{prop:block}, all facets of $P$ are blocked
and thus the controlled system (\ref{eq:odec}) is $P$-invariant. The constraints on $\theta$ also ensures,
by Lemma~\ref{lem:cond},
that for all $x\in R_X$, $h(x)\in U$.
\end{proof}

\section{Joint Synthesis of the Controller and the Invariant}

In this section, we present an iterative approach for synthesizing jointly the controller $h$ and the invariant 
polytope $P$ solving Problem~\ref{prob}.
It is based on sensitivity analysis of linear programs.
At each iteration, we use a guess for the invariant polytope $P$. Following the approach described in the
previous section, we try to synthesize a controller $h$ that renders $P$ invariant. If $P$ cannot be made
invariant by this approach, we use sensitivity analysis of linear program~(\ref{eq:inv}) to modify $P$ and obtain a new guess for 
the invariant polytope. The procedure is repeated until Problem~\ref{prob} is solved.

\subsection{Sensitivity analysis}
Let $\eta, \mu \in \R^{|\mathcal K_X|}$, let polytopes
$
P=\left\{x\in \R^n |\; \gamma_k\cdot x \le \eta_k ,\; \forall k\in \mathcal K_X \right\}$
and
$P_{\mu}=\left\{x\in \R^n |\; \gamma_k\cdot x \le \eta_k +\mu_k  ,\; \forall k\in \mathcal K_X \right\}$;
$P_\mu$ can be seen as a perturbation of polytope $P$.
The main result on sensitivity analysis is given by the following proposition:
\begin{proposition} 
\label{th:sens}
Let $d^*$ and $\left(t^*, ({\lambda^{k*}})_{k\in\mathcal K_X},({\tilde{\lambda}^{k*}})_{k\in\mathcal K_D},\theta^*\right)$ denote the optimal value and an optimal solution  
of linear program~(\ref{eq:inv}), let
$d^*_\mu$ denote the optimal value of linear program~(\ref{eq:inv}) where $\eta$ has been replaced 
by $\eta+\mu$, then
$$
 d^*_\mu \ge \displaystyle{\min_{k\in\mathcal K_X} (d^*-{\lambda^{k*}}\cdot \mu}). 
$$
\end{proposition}

\begin{proof}
For all $k\in \mathcal K_X$, for all $\overline{v_X} \in \overline{{V_X}'}$ and 
$\overline{v_D}\in  \overline{{V_D}'}$, we have:
$$
\begin{array}{ll}
& -\gamma_k \cdot (f_b(\overline{v_X},\overline{v_D})+ G_b(\overline{v_X},\overline{v_D})\theta) 
+\displaystyle{\sum_{i=1}^{|\mathcal K_X|} \displaystyle{\lambda_i^{k*}} ({\gamma_i}' \cdot \overline{v_X} - \eta_i - \mu_i)} 
+\displaystyle{\sum_{j=1}^{|\mathcal K_D|} {\tilde{\lambda}^{k*}}_j (\alpha_{D,j}'\cdot \overline{v_D}  - \beta_{D,j})} 
\\
 = & -\gamma_k \cdot (f_b(\overline{v_X},\overline{v_D})+ G_b(\overline{v_X},\overline{v_D})\theta) 
+\displaystyle{\sum_{i=1}^{|\mathcal K_X|} \displaystyle{\lambda_i^{k*}} ({\gamma_i}' \cdot \overline{v_X} - \eta_i)} 
+\displaystyle{\sum_{j=1}^{|\mathcal K_D|} {\tilde{\lambda}^{k*}}_j (\alpha_{D,j}'\cdot \overline{v_D}  - \beta_{D,j})}-{\lambda^{k*}}\cdot \mu \\
\geq &t^* -{\lambda^{k*}}\cdot \mu \geq \displaystyle{\min_{k'\in\mathcal K_X} (t^*-{\lambda^{k'*}}\cdot \mu}).
\end{array}
$$
This shows that $\left(\displaystyle{\min_{k\in\mathcal K_X} (t^*-{\lambda^{k*}}\cdot \mu}),(\lambda^{k*})_{k\in\mathcal K_X},({\tilde{\lambda}^{k*}})_{k\in\mathcal K_D},\theta^*\right)$ 
is a feasible solution for linear program~(\ref{eq:inv}) where $\eta$ has been replaced 
by $\eta+\mu$.
It follows that
$
d^*_\mu \ge \displaystyle{\min_{k\in\mathcal K_X} (t^*-{\lambda^{k*}}\cdot \mu})
$
 which leads to the expected inequality
since $d^*=t^*$.
\end{proof}

The previous result has the following implications.
Let us assume that we are not able to synthesize a controller rendering polytope $P$ invariant by solving the linear program~(\ref{eq:inv}), this means that $d^* \le 0$. Then, the previous result tells us how to 
obtain a modified polytope $P_\mu$ in order to get $d^*_\mu \ge 0$ (or at least to get  an improved $d^*_\mu \ge d^*$).
This suggests that we can solve Problem~\ref{prob} using an iterative approach described in the following paragraph.

\subsection{Iterative approach}

Initially, let us assume that we have an initial guess for the polytope $P$; one can for instance use 
$\overline P$ but other choices are possible. We use an iterative approach to solve Problem~\ref{prob}; each iteration
consists of two main steps.

\subsubsection{First step: synthesize a controller}
Given polytope $P=\left\{x\in \R^n |\; \gamma_k\cdot x \le \eta_k ,\; \forall k\in \mathcal K_X \right\}$, we use Proposition~\ref{prop:inv} to synthesize a controller $h$.
Let $d^*$ and $\left(t^*, ({\lambda^{k*}})_{k\in\mathcal K_X},({\tilde{\lambda}^{k*}})_{k\in\mathcal K_D},\theta^*\right)$ denote the optimal value and an optimal solution  
of linear program~(\ref{eq:inv}). If $d^* \ge 0$, then we found a controller rendering $P$ invariant for the
controlled system (\ref{eq:odec}) and 
Problem~\ref{prob} is solved. If $d^* <0$, then we move to the second step.

\subsubsection{Second step: modify the polytope}
We now try to find $\mu \in \R^{|\mathcal K_X|}$ ensuring that polytope
$P_{\mu}=\left\{x\in \R^n |\; \gamma_k\cdot x \le \eta_k +\mu_k  ,\; \forall k\in \mathcal K_X \right\}$ will be invariant for the controlled system (\ref{eq:odec}). For that purpose, 
Proposition~\ref{th:sens} tells us that it is sufficient that
$d^* -\lambda^{k*} \cdot \mu \ge 0$, for all $k\in \mathcal K_X$.
The requirement that $\underline P \subseteq P_{\mu} \subseteq \overline P $ 
can be translated to $\underline{\eta_k}-\eta_k \le \mu_k \le \overline{\eta}_k - \eta_k$ 
for all  $k\in \mathcal K_X$.
Also, since sensitivity analysis is pertinent mainly for small perturbations, we impose that
for all $k\in \mathcal K_X$,
$-\varepsilon \le \mu_k \le \varepsilon$ where $\varepsilon$ is a parameter that can be tuned.
Then, finding a suitable $\mu$ can be done by solving the following linear program:
\begin{equation}
\label{eq:sens1}
\begin{array}{llr}
\text{maximize} & t\\
\text{over} & t\in \R,\;  \mu \in \R^{|\mathcal K_X|},\\
\text{subject to}  
&t \le d^*-{\lambda^{k*}}\cdot \mu, &  k\in \mathcal K_X,\\
&\min(-\varepsilon,\underline{\eta_k}-\eta_k) \leq \mu_k \leq \max(-\varepsilon,\underline{\eta_k}-\eta_k), & k\in \mathcal K_X.
\end{array}
\end{equation}
Let $(t^*,\mu^*)$ be a solution of this linear program. If the optimal value $t^*$ of this problem 
is non-negative then it is sufficient to prove that the controller $h:R_X\rightarrow U$ synthesized in the first step
and polytope $P_{\mu^*}$ solve Problem \ref{prob}.
Otherwise, if $t^*<0$, then we go back to the first step and start a new iteration with $P=P_{\mu^*}$.

\begin{remark} Let us remark that the polytope $P_{\mu^*}$ computed by solving (\ref{eq:sens1}) may have empty facets.
 In order to avoid such situations, it 
may be useful to replace $\mu^*$ by $\tilde{\mu}^*$ such that $P_{\tilde{\mu}^*}$ has no empty facet 
and $P_{\mu^*}= P_{\tilde{\mu}^*}$ (see Figure~\ref{fig:constraints}). This can be done by solving a set of linear
programs.
\begin{figure}[!h]
\begin{center}
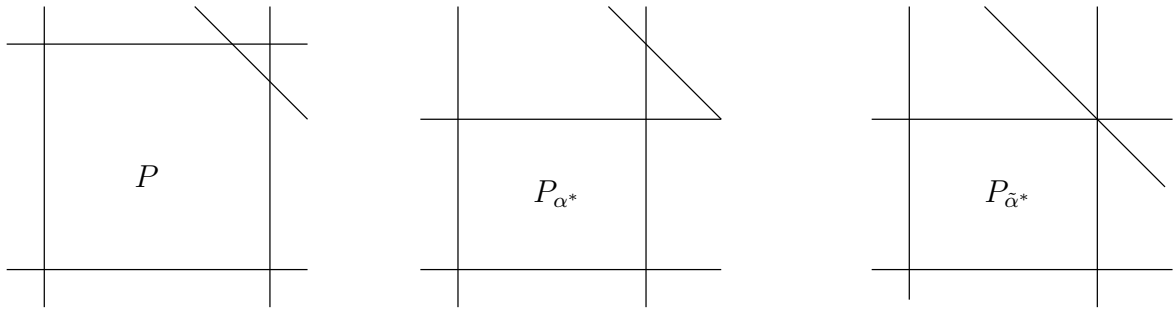
\caption{The polytope $P_{\mu^*}$ may have empty facets (center polytope), we replace $\mu^*$ by $\tilde{\mu}^*$ such that $P_{\tilde{\mu}^*}$ has no empty facet and $P_{\mu^*}= P_{\tilde{\mu}^*}$ (right polytope). \label{fig:constraints}}
\end{center}
\end{figure} 
\end{remark}

Let us discuss briefly the computational complexity of our approach. Each iteration mainly consists in solving two linear programs. The linear program~(\ref{eq:inv}) has $1+q+|\mathcal K_X| (
|\mathcal K_X| + |\mathcal K_D| )$ variables and 
$|\mathcal K_X|(|\mathcal K_X|+|\mathcal K_D|+|\overline{V_X'}|+|\overline{V_D'}|-1) + 
|\mathcal K_U| |\overline{V_X'}|$ inequality constraints. Let us remark that 
$|\overline{V_X'}|=(\delta_1+1)\times \dots \times (\delta_n+1)$ and 
$|\overline{V_D'}|=(\rho_1+1)\times \dots \times (\rho_n+1)$. Since the complexity of linear programming is polynomial
in average in the number of variables and constraints. It follows that the first step of the iteration has polynomial
cost in the number of constraints of polytopes $P$, $D$ and $U$ and in the degrees of the polynomials. The linear program~(\ref{eq:sens1})
has $1+|\mathcal K_X|$ variables and $2 |\mathcal K_X|$ inequality constraints.
It follows that the second step of the iteration has polynomial
cost in the number of constraints of polytope $P$.

\section{Examples} 

Our approach was implemented in Matlab; in the following, we show 
its application to a set of examples. 

\subsection{Moore-Greitzer jet engine model}

We tested our approach on the following polynomial system, corresponding to a Moore-Greitzer model of a jet engine~\cite{Krstic1995}:
\begin{equation}
\label{eq:moore}
\left\{
\begin{array}{rcll}
 \dot{x}_1 &=& -x_2-\frac{3}{2}x_1^2-\frac{1}{2}x_1^3 + d, \\
 \dot{x}_2 &=& u.
\end{array}\right.
\end{equation}
We first work in the rectangle $R_X=[-0.2,0.2]^2$ with disturbance $d \in D=R_D=[-0.02,0.02]$.
We want to synthesize a linear controller i.e $h(x_1,x_2)=\theta_1 x_1+\theta_1 x_2$ such that $h(x) \in U= [-0.35,0.35]$, for all $x \in R_X$.
Let $\underline{P}$ and $\overline{P}$ be polytopes with $m=24$ facets with uniformly distributed orientations
and tangent to the circles of center $(0,0)$ and of radius $0.01$ and $0.2$, respectively.
Using our approach, we found the controller $h(x_1,x_2)=0.8076 x_1-0.9424 x_2$ rendering
the polytope $P$ shown on the left part of Figure~\ref{fig:ex1} invariant.
We make a second experiment, working in the rectangle $R_X=[-0.2,0.2]^2$ with disturbance $d \in D=R_D=[-0.025,0.025]$.
We want to synthesize a polynomial control whose degrees are $3$ in $x_1$ and $1$ in $x_2$ (i.e. the same as the vector field).
$\underline{P}$ and $\overline{P}$ are polytopes with $m=8$ facets with uniformly distributed orientations
and tangent to the circles of center $(0,0)$ and of radius $0.01$ and $0.2$, respectively.
Using our approach, we found a controller rendering
the polytope $P$ shown on the right part of Figure~\ref{fig:ex1} invariant.
The previous experiments show that by looking for controller of higher degrees, we may be able to find simpler invariants for larger disturbances.

\begin{figure}[!h]
\begin{center}
\includegraphics[angle=0,scale=0.4]{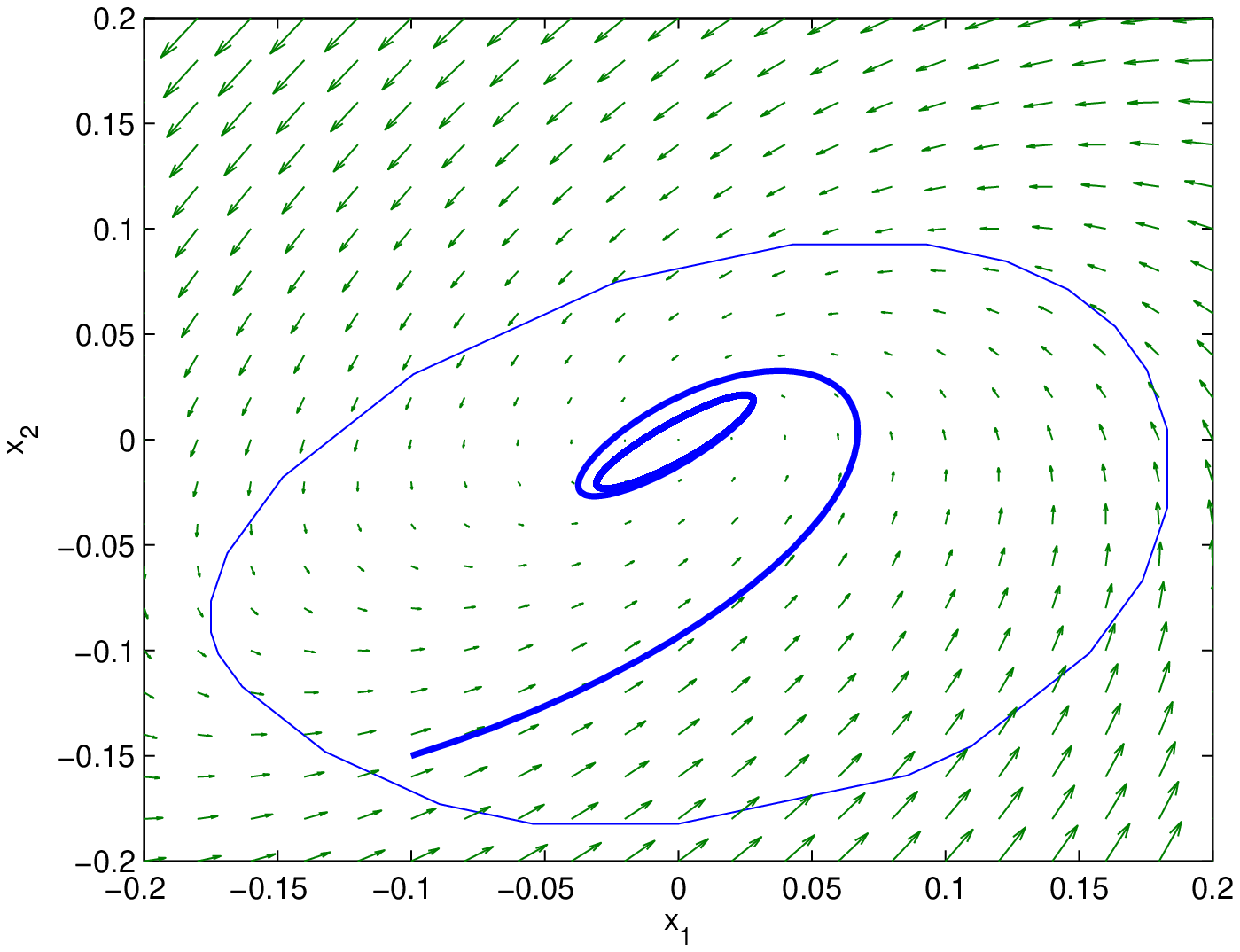} 
\includegraphics[angle=0,scale=0.4]{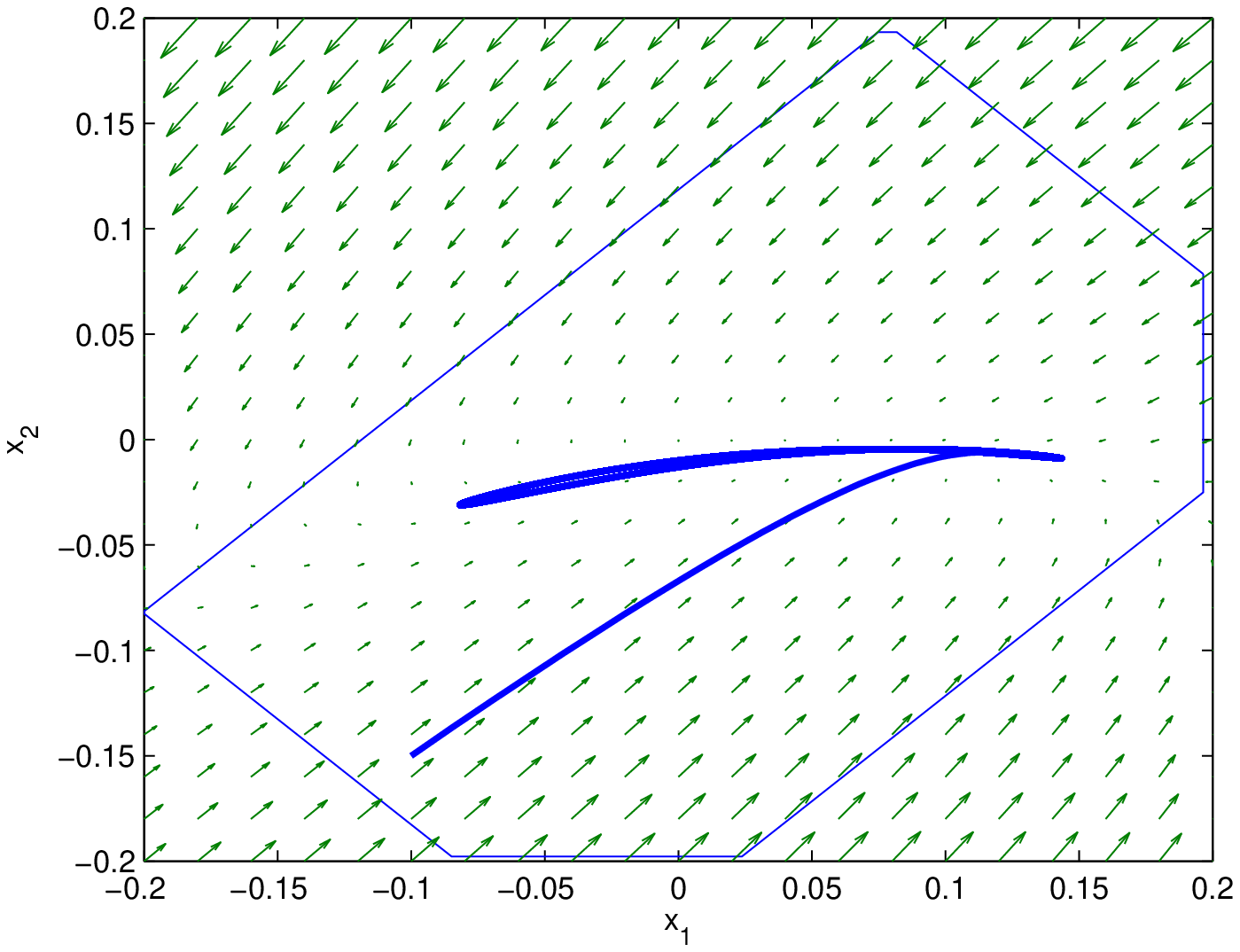} 
\caption{Left: Invariant polytope $P$ with $24$ facets and a trajectory of (\ref{eq:moore}) illustrating the invariance for disturbance $d(t)=0.02\cos(0.5t)$. Right: Invariant polytope $P$ with $8$ facets and a trajectory of (\ref{eq:moore}) illustrating the invariance for disturbance $d(t)=0.025\cos(0.1t)$.}
\label{fig:ex1}
\end{center}
\end{figure}

\subsection{Unicycle model}

We now consider a simple model of a unicycle:
$$
\left\{
\begin{array}{rcll}
 \dot{x} &=& v \cos(\varphi),\\
 \dot{y} &=& v \sin(\varphi),\\
 \dot{\varphi} &=& \omega.
\end{array}\right.
$$
where $v$ and $\omega$ are the inputs of the system representing respectively the velocity and the angular
velocity of the particle. In the following, we shall consider $v$ as a disturbance and $\omega$ as the control
input. Using the change of coordinates
$z_1 = x \cos(\varphi)+ y \sin(\varphi)$, and
$z_2 = x \sin(\varphi)- y \cos(\varphi)$, we obtain the following polynomial system.
\begin{equation}
\label{eq:unicycle}
\left\{
\begin{array}{rcll}
 \dot{z}_1 &=& v-z_2 \omega, \\
 \dot{z}_2 &=&  z_1\omega.
\end{array}\right.
\end{equation}
We work in the rectangle $R_X=[-0.1,0.1]\times [0.9,1.1]$ with disturbance $v\in D=R_D=[0.96,1.04]$. We want to
synthesize
an affine controller $h(z_1,z_2)=\theta_0+ \theta_1 z_1 + \theta_2 z_2$. In this example, we do not impose constraints on the value of the input. 
$\underline{P}$ and $\overline{P}$ are defined as polytopes with $m=24$ facets with uniformly distributed orientations
and tangent to the circles of center $(0,1)$ and radius $0.01$ and $0.1$ respectively.
Using our approach, we found the controller $h(z_1,z_2)=1.0178+1.8721 z_1-0.0253 z_2$ rendering
the polytope shown on Figure~\ref{fig:ex2} invariant.

\begin{figure}[!h]
\begin{center}
\includegraphics[angle=0,scale=0.6]{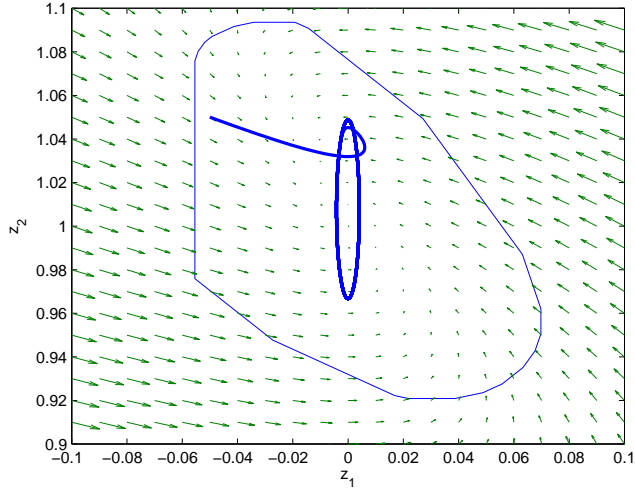} 
\caption{Invariant polytope $P$ with $24$ facets and a trajectory
of (\ref{eq:unicycle}) illustrating the invariance for disturbance $v(t)=1+0.04\cos(0.1t)$.}
\label{fig:ex2}
\end{center}
\end{figure}

\subsection{Rigid body motion}

The last example is a model describing the motion of a rigid body. 
It is borrowed from~\cite{Byrnes89}:
\begin{equation}
\label{eq:rigid}
\left\{
\begin{array}{rcll}
 \dot{x}_1 &=& u_1,\\
 \dot{x}_2 &=& u_2, \\
 \dot{x}_3 &=& x_1x_2.
\end{array}\right.
\end{equation}
We work in the rectangle $R_X=[-0.2,0.4]\times[-0.2,0.2]\times[-0.2,0.4]$. In this example, we do not consider disturbances. 
We want to synthesize a multi-affine controller i.e $h: \R^3 \rightarrow \R^2$ (defined by sixteen parameters),
such that $h(x) \in U= [-1,1]^2$, for all $x \in R_X$. 
Using our approach, we found a controller rendering
the polytope with $18$ facets, shown on Figure~\ref{fig:ex3}, invariant.
\begin{figure}[!h]
\begin{center}
\includegraphics[angle=0,scale=0.6]{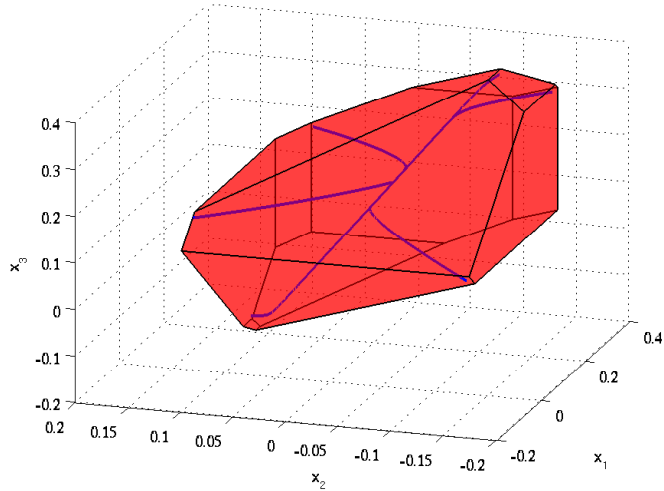} 
\caption{Invariant polytope $P$ with $18$ facets and trajectories
of (\ref{eq:rigid}) illustrating the invariance.}
\label{fig:ex3}
\end{center}
\end{figure}

\section{Conclusion}

In this paper, we have considered the problem of synthesizing controllers ensuring robust invariance of polynomial dynamical systems.
Using the recent results of~\cite{Bensassi2010} on polynomial optimization over bounded polytopes, we have developed an iterative approach to solve this problem. It is mainly based on linear programming and therefore it is effective.
We have shown applications to several examples which shows the usefulness of the approach. 
Future work will focus on a deeper theoretical analysis of the properties of the linear programming relaxations
of polynomial optimization problems as well as their application to other classes of problems in control.
 
\bibliographystyle{alpha}
\bibliography{ref}
\end{document}